\documentclass[11pt,a4paper,reqno, draft]{amsart}
\usepackage{amsmath, amssymb, latexsym, enumerate,  graphicx,tikz, stmaryrd, eufrak,enumitem,verbatim,bbold,MnSymbol}
\usetikzlibrary{arrows,decorations,decorations.pathmorphing,positioning}
\usepackage{mdframed}
\usepackage{thmtools}
\usepackage{mdwlist}
\usepackage[pdftex]{hyperref}
\usepackage{cleveref}
\usepackage{bbm}
\usepackage{cases}
\usepackage{array}
\usepackage{tikz-cd}
 \usepackage{todonotes}
\usepackage{phaistos}

\usepackage[hmarginratio={1:1},vmarginratio={1:1},lmargin=80.0pt,tmargin=90.0pt]{geometry}

\usepackage{tikz}
\tikzset{anchorbase/.style={baseline={([yshift=-0.5ex]current bounding box.center)}}}
\usetikzlibrary{decorations.markings}
\usetikzlibrary{decorations.pathreplacing}
\usetikzlibrary{arrows,shapes,positioning,backgrounds}
\tikzstyle directed=[postaction={decorate,decoration={markings,
    mark=at position #1 with {\arrow{>}}}}]
\tikzstyle rdirected=[postaction={decorate,decoration={markings,
    mark=at position #1 with {\arrow{<}}}}]
    
\usepackage{graphicx}
\usepackage[vcentermath]{youngtab}
\usepackage{nicefrac}

 \newlength{\baseunit}               
 \newcount{\numlines}                
\newtheorem{thm}{Theorem}

\newtheorem{theorem}[subsubsection]{Theorem}
\newtheorem{lemma}[theorem]{Lemma}
\newtheorem{prop}[theorem]{Proposition}
\newtheorem{corollary}[subsubsection]{Corollary}

\theoremstyle{definition}
\newtheorem{definition}[subsubsection]{Definition}

\newtheorem{example}[subsubsection]{Example}

\newcommand{\cO}{\mathcal{O}}
\newcommand{\dd}{\mathbf{d}}
\newcommand{\mN}{\mathbb{N}}
\newcommand{\mZ}{\mathbb{Z}}
\newcommand{\mC}{\mathbb{C}}
\newcommand{\fg}{\mathfrak{g}}
\newcommand{\fb}{\mathfrak{b}}
\newcommand{\fh}{\mathfrak{h}}
\newcommand{\Le}{\mathcal{L}}
\newcommand{\End}{\mathrm{End}}
\newcommand{\Hom}{\mathrm{Hom}}
\newcommand{\VX}{V\hspace{-1mm}X}
\newcommand{\cG}{\mathcal{G}}
\newcommand{\op}{\mathrm{op}}
\newcommand{\Ind}{\mathrm{Ind}}

\begin{document}
\title[Blocks in category $\cO$]{The classification of blocks in BGG category $\cO$}
\author{Kevin Coulembier}
\address{K.C.: School of Mathematics and Statistics, University of Sydney, F07, NSW 2006, Australia}
\email{kevin.coulembier@sydney.edu.au}



\begin{abstract}
We classify all equivalences between the indecomposable abelian categories which appear as blocks in BGG category $\cO$ for reductive Lie algebras. Our classification implies that a block in category $\cO$ only depends on the Bruhat order of the relevant parabolic quotient of the Weyl group. As part of the proof, we observe that any finite dimensional algebra with simple preserving duality admits at most one quasi-hereditary structure.
\end{abstract}

\maketitle


\section*{Introduction}
Fix a reductive Lie algebra $\fg$ over $\mC$ with Cartan and Borel subalgebra $\fh\subset \fb$. The abelian category $\cO(\fg,\fb)$ of $\fg$-modules associated to this data was introduced by Bernstein, Gelfand and Gelfand in \cite{BGG}. The simple modules are labelled by $\fh^\ast$. The indecomposable integrable blocks are described by orbits of the Weyl group $W=W(\fg:\fh)$ in $\fh^\ast$, and non-integrable blocks by orbits of the relevant integral Weyl subgroups. In \cite{So}, Soergel proved that a block in $\cO(\fg,\fb)$ is, up to equivalence, determined by the data of the relevant integral Weyl subgroup $U< W$ (viewed as a Coxeter system) and the parabolic subgroup $U'<U$ which `stabilises' the block. In particular, this shows that every non-integral block is equivalent to an integral block of a different Lie algebra. Furthermore, we can unambiguously write $\cO(U,U')$ to denote the block.

It is known that there are more equivalences between blocks in $\cO$ than described by Soergel's theorem. A trivial example is given by the maximally singular block $\cO(W,W)$, which is equivalent to the category of vector spaces, for each Weyl group $W$. More refined examples are listed in Theorem~2 below. Our main result is the following theorem.
\begin{thm}
Consider finite Weyl groups $W,U$ with parabolic subgroups $W'<W$ and $U'<U$. The categories $\cO(W,W')$ and $\cO(U,U')$
 are equivalent if and only if the partially ordered sets $(W/W',\le_B)$ and $(U/U',\le_B)$ are isomorphic, with $\le_B$
the Bruhat order.\end{thm}
Of course, if we have an isomorphism of Coxeter groups $W\stackrel{\sim}{\to}U$ which maps $W'$ to $U'$, the two Bruhat orders are isomorphic and we recover Soergel's theorem. All non-trivial isomorphisms of Bruhat orders correspond to the list in Theorem~2 below.

Unfortunately, we only have a conceptual proof of one direction of Theorem~1, contrary to the results in \cite{So}. Concretely, we prove that any finite dimensional algebra with simple preserving duality admits at most one quasi-hereditary structure. This applies to blocks of category $\cO$ and hence any equivalence must be one of highest weight categories. By the BGG theorem, the Bruhat order is an invariant of the highest weight structure of $\cO$. Note that this actually implies a stronger statement than the one in Theorem~1, namely any equivalence between two blocks in $\cO$ induces an isomorphism of Bruhat orders.

To prove the other direction of Theorem~1 we investigate to which extend we can recover the pair $(W,W')$, for an arbitrary Coxeter group $W$ with parabolic subgroup $W'$, from the poset $(W/W',\le_B)$. A lot of information can be reconstructed from general methods. For instance, we show that for label free simple Coxeter graphs, any Coxeter pair can be reconstructed from the Bruhat order on its parabolic quotient. Applying this and other methods to finite Weyl groups shows that isomorphisms between such posets are extremely rare. In all cases where there exists such an isomorphism, the blocks of category $\cO$ are known to be equivalent.

In order to classify blocks in $\cO$, we can restrict to simple Lie algebras, see Theorem~\ref{Thm4} below. The following thus yields a complete classification of blocks in category~$\cO$.

\begin{thm}
Consider two irreducible finite Weyl groups $W,U$ with parabolic subgroups $W'<W$ and $U'<U$. All non-trivial equivalences $\cO(W,W')\simeq\cO(U,U')$, excluding the cases $W=W'$ and $U=U'$, are:
\begin{enumerate}
\item $\cO(A_{2n+1},A_{2n})\;\simeq\; \cO(B_{n+1},B_n),\;\;$ with $n\ge 2$;
\item $\cO(B_n, A_{n-1})\;\simeq\; \cO(D_{n+1},A_n),\;\;$ with $n\ge 3$;
\item $\cO(A_3,A_2)\;\simeq\; \cO(B_2,A_1)$;
\item $\cO(A_5,A_4)\;\simeq\; \cO(G_2,A_1)\;\simeq \cO(B_3,B_2)$.
\end{enumerate}
\end{thm}

The paper is organised as follows. In Section~\ref{SecPrel} we recall the necessary background. In Section~\ref{SecQH} we prove the uniqueness of quasi-hereditary structures on algebras with simple preserving duality and mention an application to cellular algebras. In Section \ref{SecRecover} we study the reconstruction of Coxeter pairs from Bruhat orders, culminating in a classification of finite Weyl group pairs with isomorphic Bruhat order. In Section~\ref{SecO} we apply all of the above to prove Theorems 1 and 2 above. We also mention some potential applications in the study of category $\cO$ for Kac-Moody algebras and Lie superalgebras.

\section{Preliminaries}\label{SecPrel}
We set $\mN=\{0,1,2,\ldots\}$ and $\mN^\infty=\mN\cup \{\infty\}$. We also use $\mZ_{>0}=\{1,2,3,\ldots\}$ and $\mZ_{>0}^\infty=\mZ_{>0}\cup \{\infty\}$.
Whenever we will refer to the cardinality $|E|$ of some set $E$, the latter will be countable. Hence we can unambiguously consider $|E|$ as an element of $\mN^\infty$.

\subsection{Partial orders}

\subsubsection{} Consider a poset $(X,\le)$. 
We say that $a\in X$ is a greatest element if $b\le a$ for every $b\in X$. We say that $a\in X$ is a maximal element if $a\le b$ for $b\in X$ implies $b=a$. We also have the dual notions of least and minimal elements.
For $a,b\in X$, we say that $a$ covers $b$ if $b< a$ and $b\le c\le a$ implies $c\in\{a,b\}$. Then we write $b\lhd a$.
We will always use the same notation for the partial order $\le$ restricted to a subset $Y\subset X$.
 
 \subsubsection{Pointed posets}
A graded poset is a poset $(X,\le)$ is equipped with a rank function $\rho:X\to\mN$ such that 
$$b\le a\mbox{ implies }\rho(b)\le \rho(a)\quad\mbox{and}\quad b\lhd a\mbox{ implies }\rho(b)=\rho(a)-1.$$
We make the additional normalisation assumption that $\rho^{-1}(0)$ is not empty and write $X_i=\rho^{-1}(i)$, for all $i\in\mN$.
The length of a graded poset is defined as
$$\Le(X,\le)\;=\; \sup\{i\in\mN\,|\, X_n\not=0\}\;\in\,\mN^\infty.$$
For example, a finite linear poset $(X,\le)$ can be graded in the obvious way and we have $\Le(X,\le)=|X|-1$.

If a graded poset $(X,\le,\rho)$ has a least element then clearly $X_0$ is the singleton containing the least element. Furthermore, the rank function $\rho$ is then uniquely determined by our convention. We call a poset {\bf pointed} if it has a least element and admits a rank function.

\subsection{Quasi-hereditary algebras}
Let $A$ be a finite dimensional associative algebra over a field $k$.

\subsubsection{}Let $\Lambda$ denote the set of isomorphism classes of simple left modules of $A$ and denote by $L(\lambda)$ a representative of $\lambda\in\Lambda$. The category of finite dimensional left $A$-modules is denoted by $A$-mod. We denote the projective cover and injective hull of $L(\lambda)$ in $A$-mod by $P(\lambda)$ and $I(\lambda)$. We use the numbers 
$$d_\lambda\,:=\,\dim_k\End_A(L(\lambda))\in\mZ_{>0},\quad\mbox{for all $\lambda\in\Lambda$.}$$

 \subsubsection{}When $A$ is considered with a partial order $\le$ on $\Lambda$, we write $(A,\le)$. For such a partial order, following \cite{DR}, for each $\lambda\in\Lambda$ the {\bf standard module} $\Delta(\lambda)$ is defined as the maximal quotient of $P(\lambda)$ such that all simple constituents are of the form $L(\mu)$ with $\mu\le\lambda$. Dually, the costandard module $\nabla(\lambda)$ is the maximal submodule of $I(\lambda)$ with the same condition on simple constituents. 
 
 We say that $(A,\le)$ is quasi-hereditary if its module category is a highest weight category for $\le$. Concretely, this means the following.
\begin{definition}[\cite{CPS}]\label{defqh}
The algebra $(A,\le)$ is quasi-hereditary if for all $\lambda\in\Lambda$:
\begin{enumerate}
\item We have $[\Delta(\lambda):L(\lambda)]=1$.
\item The module $P(\lambda)$ has a filtration with each quotient isomorphic to a standard module $\Delta(\mu)$ with $\lambda\le\mu$.
\end{enumerate}
\end{definition}
If $(A,\le)$ is quasi-hereditary, then \cite[Lemma~2.5]{DR} implies we have the equality
$$(P(\lambda):\Delta(\mu))d_\mu\;=\;[\nabla(\mu):L(\lambda)]d_\lambda,\quad\mbox{for all $\lambda,\mu\in\Lambda$}.$$

\subsubsection{}
If $(A,\le)$ is quasi-hereditary, then so is $(A,\le')$ for every extension $\le'$ of $\le$. Moreover, the (co)standard modules for the two partial orders are identical. This motivates the following definition.
\begin{definition}\label{defess}
Let $(A,\le)$ be quasi-hereditary. We define the essential partial order $\le^e$ of $\le$ on $\Lambda$ as the partial order transitively generated by the following two relations. For $\lambda,\mu\in\Lambda$, we have $\mu\le^e\lambda$ when $[\Delta(\lambda):L(\mu)]\not=0$ or $(P(\mu):\Delta(\lambda))\not=0$.
\end{definition}
Clearly $\le$ is an extension of $\le^e$. We say that two quasi-hereditary structures $(A,\le_1)$ and $(A,\le_2)$ are equivalent if they generate the same essential partial order. This is the same as demanding their standard modules be
 identical.
 
 \begin{example}
 If $A$ is semisimple, $(A,\le)$ is quasi-hereditary for every partial order $\le$ on $\Lambda$. We then always find $\mu\le^e\lambda$ if and only if $\mu=\lambda$. In particular, all quasi-hereditary structures are equivalent.
 \end{example}

\subsubsection{} Now assume that $A$-mod has an involutive contravariant autoequivalence $\dd$ inducing the identity on $\Lambda$. We will simply say that `$A$ has a
simple preserving
duality $\dd$'. It follows that $\dd P(\lambda)\simeq I(\lambda)$ and, for every partial order, $\dd\Delta(\lambda)\simeq \nabla(\lambda)$. In particular, if $(A,\le)$ is quasi-hereditary we find
\begin{equation}
\label{BGG}
(P(\lambda):\Delta(\mu))d_\mu\;=\;[\Delta(\mu):L(\lambda)]d_\lambda,\quad\mbox{for all $\lambda,\mu\in\Lambda$}.
\end{equation}
Consequently, the essential order of $\le$ of Definition~\ref{defess} is in this case generated by the relation $\mu\le^e\lambda$ when $[\Delta(\lambda):L(\mu)]\not=0$.
Equation~\eqref{BGG} also implies that
\begin{equation}
\label{LL}
[P(\lambda):L(\lambda)]\;=\;\sum_{\mu}\frac{d_\lambda}{d_\mu}[\Delta(\mu):L(\lambda)]^2\;=\;\sum_{\mu}\frac{d_\mu}{d_\lambda}(P(\lambda):\Delta(\mu))^2.
\end{equation}

\subsection{Coxeter groups}

\subsubsection{}A Coxeter matrix consists of a countable set $S$ and a symmetric function $m:S\times S\to\mZ_{>0}^\infty$ with $m(s,s')=1$ if and only if $s=s'$. To a Coxeter matrix we associate the Coxeter group $W$ with generating set $S$ and relations $(ss')^{m(s,s')}=e$, for all $s,s'\in S$. Whenever we mention a Coxeter group $W$, it is considered as a group together with its defining set of generators $S$, {\it i.e.} as a Coxeter system $(W,S)$. We will freely use the standard results on Coxeter groups from \cite[Section~1]{BB}.   

We have the length function $\ell:W\to\mN$ of \cite[Section~1.4]{BB} which satisfies in particular $\ell(e)=0$ and $\ell(s)=1$ for $s\in S$.

Following \cite[Section~1.1]{BB}, we can graphically represent a Coxeter matrix as a Coxeter graph. In Appendix~\ref{App} we list those graphs for all finite Weyl groups. Of two vertices $s,t$ in the Coxeter graph ({\it i.e.} two elements in $S$) we say that they are neighbours if there is an edge between them, which is when $m(s,t)>2$, which is in turn equivalent with $st\not=ts$ in $W$. 

\subsubsection{Coxeter pairs and bw-Coxeter graphs}  Following \cite[Section~2.4]{BB}, for a subset $J\subset S$ we let $W_J$ be the `parabolic subgroup' of $W$ generated by the set $J$. By a {\bf Coxeter pair} we will mean a pair $(W,W_J)$ of a Coxeter group $W$ (viewed as a Coxeter system) and a parabolic subgroup $W_J<W$. Two Coxeter pairs $(W,W_J)$ and $(U,U_K)$ are isomorphic if there exists an isomorphism $\phi:W\stackrel{\sim}{\to}U$ of Coxeter systems with $\phi(W_J)=U_K$. For example, all the Coxeter pairs of the form $(A_n,A_{n-1})$ are isomorphic, whereas those of the form $(A_n,A_{n-2})$ yield two isomorphism classes.
For a Coxeter pair $(W,W_J)$, we have the set $W^J\subset W$ of shortest (with respect to $\ell$) representatives in $W$ of the parabolic quotient $W/W_J$. 
For $K\subset J\subset S$, we simplify the notation $(W_J)^K$ to $W_{J}^K$.

The information of a Coxeter pair $(W,W_J)$ can be graphically represented by the Coxeter graph of $(W,S)$, where the vertices in $J$ are white and the remaining ones black. We call such a graph a {\bf bw-Coxeter graph}. Examples are given in \ref{BU} and the proof of Theorem~\ref{Thm3} below.

\subsubsection{Bruhat orders} The Bruhat order $\le_B$ on $W$ is given in \cite[Definition~2.1.1]{BB}. We will usually employ the equivalent definition in \cite[Corollary~2.2.3]{BB}. By \cite[Theorem~2.5.5]{BB}, the poset $(W^J,\le_B)$ is pointed, for any $J\subset S$. Namely, it is graded with rank function $\ell$ and has least element $e$. If $|W|<\infty$, the length of such a poset is given by
$$\Le(W^J,\le_B)\;=\;\ell(w_0w_0^J)\;=\;\ell(w_0)-\ell(w_0^J),$$
with $w_0$ and $w_0^J$ the longest elements of $W$ and $W_J$.

\subsubsection{Coxeter factors}\label{factors} A Coxeter graph naturally decomposes into connected components. A connected graph is known as irreducible and we will likewise refer to Coxeter groups with connected graph as `irreducible Coxeter groups'. For each Coxeter system $(W,S)$ where the graph has finitely many connected components, say $S=\sqcup_iS_i$, we have a canonical  group isomorphism
$W\simeq\prod_iW_i$,
with $W_i:=W_{S_i}$. For $J\subset S$ we write $J_i=J\cap S_i$. When $J_i=S_i$, the factor $W_i$ plays no role in $W^J$. It is therefore convenient to leave out those terms in the factorisation and write
$$W\;\stackrel{J}{=}\;\prod_{i=1}^n W_i,\quad\mbox{if we have}\quad W\,=\, \prod_{i=1}^mW_i\;\mbox{ with $W_i\subset W_J$ if and only if $n<i\le m$.}$$

\subsection{Category $\cO$}\label{CatO}
In this section, we work over the field $\mC$ of complex numbers. We recall some facts about the BGG category $\cO$ from \cite{BGG}. Since we will only need very specific facts later on, we do not give any details.
\subsubsection{}\label{DefO}
For a reductive Lie algebra $\fg$ with Cartan and Borel subalgebra $\fh\subset\fb\subset\fg$, category $\cO=\cO(\fg,\fb,\fh)$ is the full subcategory of the category of finitely generated $\fg$-modules containing the modules which are locally $U(\fb)$-finite and semisimple as $\fh$-modules. The simple modules in $\cO$ are labelled by $\fh^\ast$. The block decomposition of $\cO$ is given in \cite[Section~1.13]{Hu}. For an integral weight $\lambda\in\fh^\ast$, see \cite[Section 0.6]{Hu}, the (isoclasses of) simple modules in the block containing $\lambda$ are labelled by the orbit $W\cdot\lambda$ for $W=W(\fg,\fh)$ the Weyl group of $\fg$. This Weyl group is a Coxeter group with generators given by reflections with respect to simple roots of $\fb$. Now take $\lambda$ to be anti-dominant, see \cite[Section~3.5]{Hu}, and $W_J<W$ the stabiliser of $\lambda$. The block containing $\lambda$ is denoted by $\cO_\lambda$. We thus find a  bijection between $W^J$ and the set of simple modules in $\cO_\lambda$, given by $w\mapsto w\cdot\lambda$. Non-integral blocks can be described similarly.

It is shown in \cite[Theorem~11]{So} that, up to equivalence of categories, $\cO$ only depends on (the isomorphism class of) the Coxeter pair $(W,W_J)$. We will therefore write $\cO(W,W_J)$ instead of $\cO_\lambda(\fg,\fb,\fh)$.

\subsubsection{}\label{ResSoergel}By \cite[Proposition~3.13]{Hu}, a block in $\cO$ is equivalent to $A$-mod for a finite dimensional algebra, and by \cite[Theorem~3.2]{Hu} it has a simple preserving duality.
 By \cite[Theorem~3.10]{Hu}, the algebra $A$ is quasi-hereditary for the partial order $\le$ as defined in \cite[Section~0.6]{Hu}. Moreover, it is shown in \cite[Section~5.1]{Hu} that the essential order $\le^e$ is the Bruhat order $\le_B$ on $W^J$. Actually, in this special case, the partial order $\le^e$ only contains the generating relation(s) in Definition~\ref{defess}.

\section{Quasi-hereditary algebras with simple preserving duality}\label{SecQH}

Consider a finite dimensional algebra $A$ over a field $k$.

\subsection{Uniqueness of highest weight structure}\label{SecUnique}
The main result of this section states that when $A$ has a simple preserving duality, it admits at most one quasi-hereditary structure, up to equivalence.
\begin{theorem}\label{ThmQH}
Assume $A$ has a simple preserving duality. If for two partial orders $\le_1$ and $\le_2$, both $(A,\le_1)$ and $(A,\le_2)$ are quasi-hereditary, then $\le_1^e=\le_2^e$. Equivalently, the standard modules for both partial orders coincide.
\end{theorem} 
\begin{proof}
We denote the standard modules with respect to the partial orders $\le_1$ and $\le_2$ by $\Delta_1(\lambda)$ and $\Delta_2(\lambda)$. We will start by proving that $[\Delta_1(\lambda)]=[\Delta_2(\lambda)]$, in the Grothendieck group $K_0(A)$, by induction on $\lambda$ along $\le_1$. 

First assume that $\lambda$ is maximal in $\le_1$. By Definition~\ref{defqh}(2), we find $P(\lambda)=\Delta_1(\lambda)$. By Definition~\ref{defqh}(1), this means in particular that $[P(\lambda):L(\lambda)]=1$. Hence, \eqref{LL} applied to $(A,\le_2)$ implies $P(\lambda)=\Delta_2(\lambda)$. In particular, we find $\Delta_1(\lambda)=\Delta_2(\lambda)$.

Now we fix $\lambda\in\Lambda$ such that for all $\lambda<_1\nu$ we already know that $[\Delta_1(\nu)]=[\Delta_2(\nu)]$. In particular, we define
$$n_\nu:=[\Delta_1(\nu):L(\lambda)]=[\Delta_2(\nu):L(\lambda)]\qquad\mbox{for $\lambda<_1\nu$.}$$
By \eqref{LL} we have
$$1+\sum_{\lambda<_1\nu}\frac{d_\lambda n_\nu^2}{d_\nu}\;=\; [P(\lambda):L(\lambda)]\;=\;1+\sum_{\lambda<_1\nu}\frac{d_\lambda n_\nu^2}{d_\nu}+\sum_{\lambda\not\le_1\mu} \frac{d_\lambda}{d_\mu}[\Delta_2(\mu):L(\lambda)]^2.$$
We thus find $[\Delta_2(\mu):L(\lambda)]=0$ unless $\lambda\le_1\mu$.
 By \eqref{BGG} we find that in $K_0(A)$ we have
$$[\Delta_1(\lambda)]\;=\;[P(\lambda)]-\sum_{\lambda<_1\nu}\frac{d_\lambda n_\nu}{d_\nu}[\Delta_1(\nu)]\;=\;[\Delta_2(\lambda)].$$
This concludes the proof that $[\Delta_1(\lambda)]=[\Delta_2(\lambda)]$ for all $\lambda\in\Lambda$. By definition of the essential order, this shows $\le_1^e=\le_2^e$.
\end{proof}

\subsection{Examples and applications}

\subsubsection{Cellular algebras} Assume that $A$ has a cell datum as in \cite[Definition~1.1]{GL}, {\it i.e.} $A$ is cellular. Part of the cell datum is an anti-involution $\star$ of the $k$-algebra $A$, which thus defines a duality on $A$-mod. By \cite[Definition~1.1(C2)]{GL} and \cite[Theorem~3.4]{GL}, this duality is simple preserving. 

In \cite[Theorem~1.1]{KX} it is proved that the cell datum on $A$ induces a quasi-hereditary structure on $A$ if and only if the global dimension of $A$ is finite. This shows in particular that if the cell datum does not lead to a quasi-hereditary structure, $A$ does not admit any quasi-hereditary partial order. Together with Theorem~\ref{ThmQH}, this thus implies that the only quasi-hereditary structure a cellular algebra can admit comes from its cell datum.

This observation is for instance applicable to (walled) Brauer algebras, Temperley-Lieb algebras and partition algebras, see~\cite{GL, KX}.

\subsubsection{Category $\cO$}\label{CorO}  
Recall the conclusions from Section~\ref{CatO}.
Assume that we have a $\mC$-linear equivalence
$$\cO(W,W_J)\;\stackrel{\sim}{\to}\; \cO(U,U_K)$$
for finite Weyl groups $W, U$ and parabolic subgroups $W_J<W$ and $U_{K}<U$. By Theorem~\ref{ThmQH}, this must be an equivalence of highest weight categories. In particular, considering the equivalence on simple objects must induce an order isomorphism
$$(W^J,\le_B)\;\stackrel{\sim}{\to}\; (U^K,\le_B).$$

\section{Recovering Coxeter pairs from Bruhat orders}\label{SecRecover}
Motivated by the observations in \ref{CorO} we investigate to which extend a Coxeter pair $(W,W_J)$ can be recovered from the poset $(W^J,\le_B)$.

\subsection{Coxeter systems} It is easy to see that a Coxeter system $(W,S)$ can be recovered from its Bruhat order. In this subsection we prove a slightly more general statement, namely that the poset $(W,\le_B)$ is isomorphic to some Bruhat order $(U^K,\le_B)$ if and only if $U$ and $W\times U_K$ are isomorphic as Coxeter groups.

\subsubsection{} Consider an arbitrary pointed poset $(X,\le)$. For a subset $I\subset X_1=\rho^{-1}(1)$, we define 
$$X_2(I)\;=\;\{x\in X_2\,|\, y\lhd x \mbox{ if and only if }y\in I\}.$$
For $I=\{a_1,\ldots,a_n\}$ we simplify notation as $X_2(\{a_1,\ldots,a_n\})$ to $X_2(a_1,\ldots,a_n)$. We also set
$$X^0(I)\;=\;X_0\sqcup I \sqcup\{x\in X_{>1}\,|\, y\le x \mbox{ with }y\in X_2\mbox{ implies }y\in X_2(I)\}$$
and
$$X^\infty(I)\;=\;\{x\in X\,|\, y\le x \mbox{ with }y\in X_1\mbox{ implies }y\in I\}.$$
The subsets $X^0(I)\subset X^\infty(I)$ of $X$ inherit the structure of a pointed poset. Note that for $a\in X_1$, we have $X^\infty(a)=X^0(a)$. In general, we have $X^\infty(I)\cap X_1=I$ and $X^0(I)\cap X_2=X_2(I)$.

\begin{definition}\label{DefTrip}
To a pointed poset $(X,\le)$ we associate the triple $(X_1, \mu,\nu)$ comprising the set $X_1$ with two functions
$$\mu:X_1\times X_1\to\mZ_{>0}^\infty,\;(a,b)\mapsto \Le(X^0(a,b))\quad\mbox{and}\quad \nu:X_1\to\mN^\infty,\; a\mapsto |X_2(a)|.
$$

\end{definition}

\begin{theorem}\label{Thm1}
For a Coxeter system $(W,S)$ with matrix $m:S\times S\to\mZ_{>0}^\infty$ and a subset $J\subset S$, consider the pointed poset $(X,\le,\rho):=(W^J,\le_B,\ell)$. The triple $(X_1,\mu,\nu)$ of Definition~\ref{DefTrip} is determined by the following properties:
\begin{enumerate}
\item The subset $X_1\subset W^J$ is precisely $S\backslash J$.
\item For distinct $s_1,s_2\in S\backslash J$, we have $m(s_1,s_2)=\mu(s_1,s_2)$.
\item For $s\in S\backslash J$, the number of neighbours of $s$ in $J$ is given by $\nu(s)$.
\end{enumerate}
\end{theorem}
\begin{proof}
Parts (1) and (3) are obvious. Now we prove part (2). For $\{s_1,s_2\}\subset S\backslash J=X_1$ it is clear that $X_2({s_1,s_2})$ contains only $s_1s_2$ and $s_2s_2$, where the two elements might be identical. We will prove in the next paragraph that $X^0(s_1,s_2)= W_{\{s_1,s_2\}}.$ The length of the Bruhat order on $W_{\{s_1,s_2\}}$ is clearly $m(s_1,s_2)$.

That $W_{\{s_1,s_2\}}\subset X^0(s_1,s_2)$ is obvious. Now take $x\in X^0(s_1,s_2)$ and assume that there exists $t\in S\backslash\{s_1,s_2\}$ which appears in a reduced expression of $x$. Without loss of generality, we assume that $t$ is the right-most such reflection in the relevant reduced expression.
If $t\in J$ then $t$ cannot appear on the right in $x\in W^J$, in particular it cannot commute with both $s_1$ and $s_2$. But this means that $ts_i\in X_2$ and $ts_i\le x$ for at least on $i$, a contradiction by the definition of $X^0(s_1,s_2)$. Hence $t\not\in J$. But this means that $s_it$ and $ts_i$ are in $X_2$, for both $i$. At least one of them is under $x$ in the Bruhat order and again we find a contradiction. This proves part (2).
\end{proof}

\begin{corollary}
Assume that two Coxeter groups $W,U$ with a parabolic subgroup $U_K<U$ yield an order isomorphism $(W,\le_B)\simeq(U^K,\le_B)$. Then we have
an inclusion of Coxeter groups $W\hookrightarrow U$, which together with the canonical inclusion $U_K\hookrightarrow U$ induces an isomorphism 
$W\times U_K\stackrel{\sim}{\to} U. $
\end{corollary}
\begin{proof}Denote the Coxeter systems by $(W,S)$ and $(U,T)$.
By Theorem~\ref{Thm1}(1), an isomorphism $(W,\le_B)\stackrel{\sim}{\to}(U^K,\le_B)$ induces an injection $S\hookrightarrow T$ where the image $S'$ has complement $K$. By Theorem~\ref{Thm1}(3), $S'$ consists of vertices in the Coxeter graph which are not neighbouring to any vertex in $K$. In particular, $U\simeq U_{S'}\times U_K$ as Coxeter groups. Finally, Theorem~\ref{Thm1}(2) implies that the Coxeter matrices of $(W,S)$ and $(U_{S'},S')$ are identical, hence $W\simeq U_{S'}$ as Coxeter groups.
\end{proof}

\subsection{Irreducible factors}

In this subsection we reduce the problem of recovering Coxeter pairs from their Bruhat orders to the case of irreducible Coxeter groups.
\begin{theorem}\label{Thm2}
Consider two Coxeter systems $(W,S)$ and $(U,T)$ with $J\subset S$ and $K\subset T$. Assume the Coxeter graphs $S$ and $T$ have finitely many connected components. If the partial orders $(W^J,\le_B)$ and $(U^K,\le_B)$ are isomorphic, there exists $n\in\mN$ such that we can order the factorisations into irreducible Coxeter groups as
$$W\;\stackrel{J}{=}\;\prod_{i=1}^{n}W_i\qquad\mbox{and}\qquad U\;\stackrel{K}{=}\;\prod_{i=1}^{n}U_i$$
such that we have order isomorphisms $(W_i^{J_i},\le_B)\simeq (U_i^{K_i},\le_B)$ for all $1\le i\le n$.
\end{theorem}
By definition in \ref{factors}, the irreducible factors of $W$ and $U$ which are ignored in the theorem all satisfy $|W_j^{J_j}|=1=|U_l^{K_l}|$ and hence trivially also lead to isomorphic posets.
Before proving the theorem, we need some preparatory results and notions.

\subsubsection{}\label{leadsto} Consider again an arbitrary pointed poset $(X,\le)$. We define a binary relation $\leadsto$ on $X_1$, where $a\leadsto b$ means that $\mu(a,b)=2$ and there exists $y\in X^0(a)$ and $\{y_1,y_2\}\subset X$ with $y_1\rhd y\lhd y_2$ and $y_1\ge b\le y_2$. 
We also introduce two equivalence relations $\leftrightsquigarrow$ and $\sim$ on $X_1$. The relation $\leftrightsquigarrow$ is the minimal equivalence relation with $a \leftrightsquigarrow b$ whenever $a\leadsto b$ or $b\leadsto a$.
The relation $a\sim b$ is the minimal equivalence relation with $a\sim b$ whenever $\mu(a,b)>2$ or $b\leftrightsquigarrow a$.

\begin{prop}\label{PropIrr}
For a Coxeter system $(W,S)$ and a subset $J\subset S$, consider the pointed poset $(X,\le,\rho):=(W^J,\le_B,\ell)$. 
\begin{enumerate}
\item For $a,b\in X_1=S\backslash J$ with $\mu(a,b)=2$, we have $a\leadsto b$ if and only if there exists a path in the Coxeter graph of $(W,S)$ which connects $a$ and $b$ via vertices exclusively in $J$. In particular, the binary relation $\leadsto$ is symmetric.
\item For $a,b\in X_1$, we have $a\sim b$ if and only if they belong to the same connected component of the Coxeter graph of $(W,S)$.
\item For an equivalence class $I\in X_1/\hspace{-1mm}\sim$, let $\hat{I}$ denote the connected component in the Coxeter graph of $S$ which contains $I$. Also set $\mathring{I}=\hat{I}\backslash I=J\cap \hat{I}$.
The poset $(X^\infty(I),\le)$ is isomorphic to $(W_{\hat{I}}^{\mathring{I}},\le_B)$.
\end{enumerate}
\end{prop}
\begin{proof}
For the proof of part (1), we consider arbitrary $a,b\in X_1$ with $\mu(a,b)=2$.
Assume first that there exists $n\in\mZ_{>0}$ and distinct $r_1,r_2,\ldots, r_n\in J$ such that each two consecutive elements in the sequence $a,r_1,\ldots,r_n, b$ are neighbours. Then we can define 
$$y=r_nr_{n-1}\cdots r_1a\;\in\, X^0(a), \quad y_1=br_nr_{n-1}\cdots r_1a\mbox{ and }y_2=r_nr_{n-1}\cdots r_1ab.$$
Clearly $y_1,y_2\in W^J$ satisfy the conditions of \ref{leadsto} to conclude $a\leadsto b$. This the ``if'' direction of part (1).

To prove the other direction of part (1), we observe that every reduced expression of an element $y\in X^0(a)$ is a word of elements in the connected component of the Coxeter graph of $J\sqcup \{a\}$ which contains $a$. By Theorem~\ref{Thm1}(2), $a$ and $b$ commute. If $b$ also commutes with all elements of $J$ used in the reduced expressions of $y$, then there can be only one $z\in X$ which satisfies $y\lhd z$ and $b\le z$, namely $z=yb=by$. Hence $a\leadsto b$ implies that we can form a path as desired, which concludes the proof of part (1).

Part (2) is an immediate application of part (1) and Theorem~\ref{Thm1}(2).

Now we prove part (3). Clearly we have $W_{\hat{I}}^{\mathring{I}}\subset X^\infty(I)$. Assume there exists $x\in X^\infty(I)$ with a reduced expression containing a simple reflection $t\in S$ not in the connected component $\hat{I}$. Without loss of generality, we assume that $t$ is the right-most element in a reduced expression and thus $t\not \in J$. Hence we have $t\in X_1\backslash I $ and $t\le x$, a contradiction.
\end{proof}

\begin{proof}[Proof of Theorem~\ref{Thm2}]
An isomorphism $(X,\le)\stackrel{\sim}{\to}(Y,\le)$ between two pointed posets must restrict to an isomorphism $(X_1,\sim)\stackrel{\sim}{\to}( Y_1,\sim)$. In particular, for an equivalence class $I\in X_1/\sim$, there must be $I'\in Y/\sim$ such that the posets $X^\infty(I)$ and $Y^\infty(I')$ are isomorphic.
The conclusion thus follows from Proposition~\ref{PropIrr}(2) and (3).
\end{proof}

\subsection{Simple graphs}By Theorem~\ref{Thm2} it is justified to consider connected Coxeter graphs. Since our motivation comes from finite Weyl groups we will furthermore focus on graphs without cycles. When  a Coxeter graph satisfies these properties, {\it i.e.} the underlying graph (obtained by ignoring labels) is simply connected as a topological space, we say it is simple.

\subsubsection{} Let $(X,\le)$ be a pointed poset. We define the subset $\VX\subset X$ as
$$\VX\;=\;\{x\in X_{>0}\,|\, \mbox{for every $i\in\mN$, there is at most one $y\in X_i$ with $y\le x$}\}.$$
To a pointed poset $(X,\le)$ we associate a bw-Coxeter graph $\cG(X,\le)$ as follows:
\begin{itemize}
\item the black vertices are identified with $X_1\subset \VX$;
\item the white vertices are identified with $\VX\backslash X_1$;
\item between two distinct black vertices $\{a,b\}\subset X_1$ we write an edge if $\mu(a,b)\ge 3$, and the edge is labelled with $\mu(a,b)$ if  $\mu(a,b)> 3$;
\item for each covering $x\lhd y$ with $x,y\in \VX$ we add an edge between the corresponding vertices;
\item for $a\in X_1$ and for $a\not\le x\in \VX$ minimal in $(X,\le)$ with the property that there exist $\{x_1,x_2\}\subset X$ with $x_1\rhd x\lhd x_2$ and $x_1\ge a\le x_2$, we add an edge between $a$ and $x$.
\end{itemize}
By construction of $\cG(X,\le)$, we can only have labels on edges between two black vertices.

\begin{lemma}\label{unique}
Consider a Coxeter system $(W,S)$ with simple graph. For $J\subset S$, we set $(X,\le):=(W^J,\le_B)$. The subset $\VX\subset W^J$ consists of those $w\in W^J$ for which 
\begin{enumerate}
\item every $v\in W^J$ with $v\le_Bw$ has unique reduced expression;
\item precisely one element of $S\backslash J$ appears in reduced expressions of $w$.
\end{enumerate}
\end{lemma}
\begin{proof}
We prove the claim by induction along $\ell: W^J\to\mN$. For $X_1=S\backslash J=\ell^{-1}(1)$ there is nothing to prove.
Now consider $w\in W^J$ with $\ell(w)>1$ and assume the claim has been proved for all $u\in W^J=X$ with $\ell(u)<\ell(w)$.

First assume that $w$ satisfies (1)-(2) in the lemma. Consider $y\in W$ with $y\lhd w$. In particular, $y$ is obtained by removing one simple reflection from the reduced expression of $w$. If this was the right-most one, then by assumption (2) we find $y\not\in W^J$.
Assume that the removed reflection was not the left-most one either. Since $w$ has unique reduced expression and the Coxeter graph is simple, there are two adjacent letters in the reduced expression of $y$ inherited from $w$ which are the same or commute. In the first case we have $y\not\lhd w$, in the second case $y$ does not have unique reduced expression and hence $y\not\in W^J$ by (1). In conclusion, there is only one element in $W^J$ covered by $w$. Moreover, by construction that element satisfies (1)-(2) and is thus by induction hypothesis in $\VX$. It now follows immediately that $w\in\VX$.

Assume conversely that $w\in\VX$. Clearly (2) is satisfied. Assume that (1) is not satisfied. Consider $x\in X$ with $x\le w$ which does not have unique reduced expression, and has minimal $\ell(x)$ under those assumptions. By minimality, $x$ has two reduced expressions with a different simple reflection on the left. Omitting those first letters in the respective expressions yields two distinct $x_1,x_2\le x\le w$ with $x_1,x_2\in W^J$ and $\ell(x_1)=\ell(x_2)$, which is a contradiction.
\end{proof}

\subsubsection{}\label{BU} Consider a simple bw-Coxeter graph which only has labels on edges between two black vertices. To that input we will associate a new connected bw-Coxeter graph. The black vertices and the edges between them, including possible labels, remain the same. On the other hand, each connected component of the subgraph of white vertices in the original graph will appear $n$ times in the new graph, identically attached to the black vertices, with $n$ the number of black vertices which neighbour the connected component.

For example, for the Coxeter pair $(W,W_J)=(E_6,A_3\times A_1)$ with bw-Coxeter graph as below on the left, the above procedure would yield the bw-Coxeter graph on the right
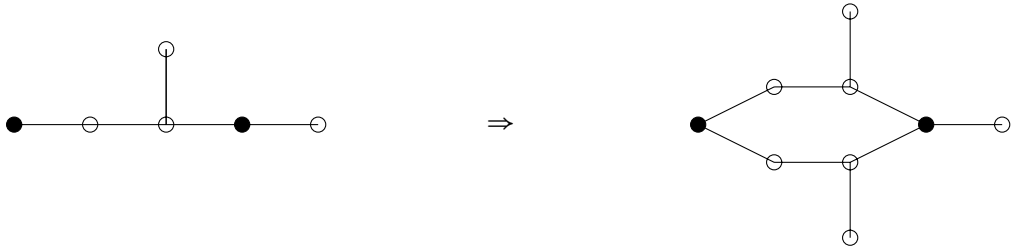
\begin{center}
\begin{tikzpicture}
\draw[fill=black] 

(0,0)                         
      circle [radius=.1]  --
(1,0) 
      --
(2,0) 
       --  
    (3,0)      
    circle [radius=.1]  --
      (4,0)      
    
    (2,0) --++(90:1)           
    
    (6.4,0) node {$\Rightarrow$}
        
    (9,0)
    circle [radius=.1] 
    
    (11,.5)
    --
    (12,0) circle [radius=.1]    
;

\draw

(1,0) 
      circle [radius=.1]
(2,0) 
      circle [radius=.1]    
      (2,0) --++(90:1)         circle [radius=.1] 
      (4,0)      
    circle [radius=.1]

    (9,0) --
    (10,0.5) circle [radius=.1] 
    --
   (11,.5) circle [radius=.1]--
   (11,1.5) circle [radius=.1]
   
    (9,0) --
    (10,-0.5) circle [radius=.1] 
    --
   (11,-.5) circle [radius=.1]--
   (11,-1.5) circle [radius=.1]
   
   (11,-.5) --
   (12,0) --
   (13,0) circle [radius=.1]
   
   (13.5,0) node {.}
    ;

\end{tikzpicture}
\end{center}
 It is clear that this procedure maps non-isomorphic bw-Coxeter graphs to non-isomorphic bw-Coxeter graphs.

\begin{theorem}\label{ThmNew}
Consider a Coxeter system $(W,S)$ with simple graph and with $J\subset S$ such that no edge which meets $J$ is labelled. The bw-Coxeter graph $\cG(W^J,\le_B)$ is obtained from the one of $(W,W_J)$ by procedure~\ref{BU}.
\end{theorem}

\begin{proof}Set $(X,\le)=(W^J,\le_B)$.
Let $V$ denote the set of pairs $(s,t)$ with $s\in S\backslash J$ and $t\in J$ such that the unique minimal path from $s$ to $t$ in the Coxeter graph of $(W,S)$ contains, besides $s$, only elements of $J$. It follows easily from Lemma~\ref{unique} that we have a bijection
$$\phi:V\;\stackrel{1:1}{\to}\; \VX\backslash X_1,\quad (s,t)\mapsto t r_d r_{d-1}\cdots r_1s.$$
where $r_i$ label the vertices in the Coxeter graph along the minimal path from $s$ to $t$ in the obvious way. In case $m(s,t)>2$, we thus have $d=0$.

By Theorem~\ref{Thm1}, the Coxeter subgraphs of $\cG(X,\le)$ and the bw-Coxeter graph of $(W,W_J)$ consisting of black vertices (and labelled edges between them) are canonically isomorphic.
By the description of $\phi$, we have an edge between $\phi(s,t)$ and $\phi(s,t')$ in $\cG(X,\le)$ if and only if there is an edge between $t$ and $t'$. By definition of $\cG(X,\le)$ there is no edge between $\phi(s,t)$ and $\phi(s',t')$ whenever $s\not=s'$. The proof of Proposition~\ref{PropIrr}(1) shows that 
there is an edge between $\phi(s,t)$ and $s'\in X_1$ in $\cG(X,\le)$ if and only if $t\in J$ neighbours $s'$.

It now follows that $\cG(X,\le)$ is obtained from the bw-Coxeter graph of $(W,W_J)$ by the procedure of \ref{BU}.
Concretely, consider a connected component $\mathring{I}$ in the white part of the graph of $(W,W_J)$ and label the neighbours of $\mathring{I}$ in $S\backslash J$ by $\{s_i\,|\, 1\le i\le n\}$. Then the subgraph $\mathring{I}$ of $S$ yields $n$ copies in $\cG(X,\le)$ and they correspond to the sets of vertices $\{\phi(s_i,t)\,|\, t\in\mathring{I}\}$, for $1\le i\le n$.\end{proof}

\subsection{Finite Weyl groups}

\begin{theorem}\label{Thm3}
The only non-isomorphic Coxeter pairs $(W,W_J)$ and $(U,U_K)$, with $W,U$ irreducible finite Weyl groups which lead to isomorphic posets $(W^J,\le_B)\simeq (U^K,\le_B)$ are 
\begin{enumerate}
\item $(A_{2n+1},A_{2n})\;\longleftrightarrow\; (B_{n+1},B_n),\;\;$ for $n\ge 2$;
\item $(B_n, A_{n-1})\;\longleftrightarrow\; (D_{n+1},A_n),\;\;$ for $n\ge 3$;
\item $(A_3,A_2)\;\longleftrightarrow\; (B_2,A_1)$;
\item $(A_5,A_4)\;\longleftrightarrow\; (G_2,A_1)\;\longleftrightarrow\; (B_3,B_2) $;
\item $(W,W)\;\longleftrightarrow\; (U,U)$, for any $U\not\simeq W$.
\end{enumerate}
\end{theorem}

We note that case (3) can be seen as a limit of both series in (1) and (2). Before proving the theorem, we establish the following lemma.

\begin{lemma}\label{LemNew}
Besides the cases listed in Theorem~\ref{Thm3}, the only non-isomorphic Coxeter pairs $(W,W_J)$ and $(U,U_K)$ with $W,U$ irreducible finite Weyl groups which lead to isomorphic bw-Coxeter graphs $\cG(W^J,\le_B)\simeq\cG(U^K,\le_B)$ are
\begin{itemize}
\item $(F_4,A_2)\;\longleftrightarrow\; (D_5,A_3)$;
\item $(F_4,B_3)\;\longleftrightarrow\; (E_6,D_5)$;
\item $(B_{m+n},P\times  A_{n-1})\;\longleftrightarrow\; (D_{m+n+1},P\times  A_{n} )$, for every parabolic subgroup $P<A_{m-1}$, using the canonical inclusions $A_{m-1}\times A_{n-1}<B_{m+n}$ and $A_{m-1}\times A_{n}<D_{m+n+1}$, for $n\ge 2$ and $m\ge 1$; 
\item $(B_{m+n}, Q\times B_{n-1})\;\longleftrightarrow\; (A_{m+2n-1}, Q\times A_{2n-2})$, for every parabolic subgroup $Q<A_{m}$, using the canonical inclusions $A_{m}\times B_{n-1}<B_{m+n}$ and $A_m\times A_{2n-2}<A_{m+2n-1}$ for $n\ge 2$ and $m\ge 1$ (where $B_{n-1}$ for $n=2$ is to be interpreted as $A_1$). 
\end{itemize}
\end{lemma}
\begin{proof}
By Theorem~\ref{ThmNew}, a necessary condition for an isomorphism $\cG(W^J,\le_B)\simeq\cG(U^K,\le_B)$ is that at least one of the original bw-Coxeter graphs of $(W,W_J)$ and $(U,U_K)$ contains a labelled edge linked to a white vertex. First we calculate $\cG(W^J,\le_B)$ for all pairs $(W,W_J)$ with such an edge and for $W\in \{F_4,G_2\}$. On the left-hand side we write the original bw-Coxeter graph of $(W,W_J)$, on the right-hand side $\cG(W^J,\le_B)$.
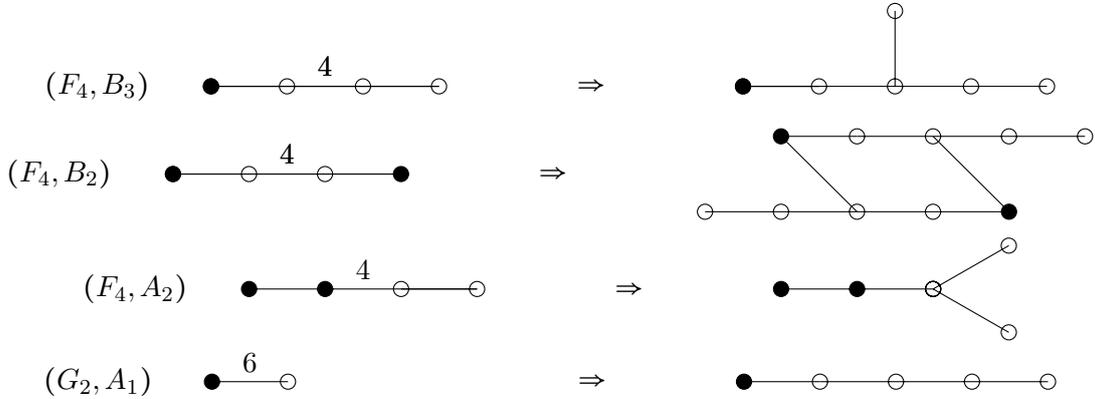
\begin{center}
\begin{tikzpicture}
\draw[fill=black] 
(-1.5,0) node {$(F_4,B_3)$}

(0,0)                         
    circle [radius=.1]  --
(1,0) 
      --node [midway,above] {$4$}
(2,0) 
      --    
(3,0)

   (5,0) node {$\Rightarrow$}

(7,0)                         
    circle [radius=.1]  --
(8,0)

    ;

\draw
(0,0)                         
      circle [radius=.1]  --
(1,0) 
      circle [radius=.1] --node [midway,above] {$4$}
(2,0) 
      circle [radius=.1] --    
(3,0)
    circle [radius=.1]      
    
    (7,0)                         
      circle [radius=.1]  --
(8,0) 
      circle [radius=.1] node [above] {} --
(9,0) 
      circle [radius=.1] node [above] {}    --  
    (10,0)      
    circle [radius=.1] node [above] {} --
      (11,0)      
    circle [radius=.1] 
    (9,0) --++(90:1)         circle [radius=.1]    
;

\end{tikzpicture}

\vspace{2mm}

\begin{tikzpicture}
\draw[fill=black] 
(-1.5,0) node {$(F_4,B_2)$}

(0,0)                         
    circle [radius=.1]  --
(1,0) 
      --node [midway,above] {$4$}
(2,0) 
      --    
(3,0)
circle [radius=.1]    

   (5,0) node {$\Rightarrow$}

(8,0.5)                         
    circle [radius=.1]  --
(9,0.5)      
         
         (10,0.5) --
         (11,-0.5)
         circle [radius=.1]  
    ;

\draw
(0,0)                         
      circle [radius=.1]  --
(1,0) 
      circle [radius=.1] --node [midway,above] {$4$}
(2,0) 
      circle [radius=.1] --    
(3,0)
      
    (9,0.5)      
    circle [radius=.1] {} --
    (10,0.5)
    circle [radius=.1] --
    (11,0.5)
    circle [radius=.1]--
    (12,0.5)
    circle [radius=.1]
    
    (8,0.5) --
    (9,-0.5) circle [radius=.1] {} --
    (10,-0.5) circle [radius=.1] {} --
    (11,-0.5)
    (7,-.5)circle [radius=.1] {} --
    (8,-.5) circle [radius=.1] {} --
    (9,-.5)
      ;

\end{tikzpicture}

\vspace{2mm}

\begin{tikzpicture}
\draw[fill=black] 
(-1.5,0) node {$(F_4,A_2)$}

(0,0)                         
    circle [radius=.1]  --
(1,0) 
      circle [radius=.1] --node [midway,above] {$4$}
(2,0) 
      --    
(3,0)

   (5,0) node {$\Rightarrow$}

(7,0)                         
    circle [radius=.1]  --
(8,0)      
        circle [radius=.1] --
        (9,0)
         
    ;

\draw
(2,0) 
      circle [radius=.1] --    
(3,0)
    circle [radius=.1]      
        (9,0)      
    circle [radius=.1] node [above] {} --
      (9,0)      
    circle [radius=.1] node [above] {}     
    
(9,0) 
      circle [radius=.1] node [above] {}      
            
(9,0) --++ (30:1.15)
      circle [radius=.1] 
      
(9,0) --++ (-30:1.15)
      circle [radius=.1]      
;

\end{tikzpicture}

\begin{tikzpicture}
\draw[fill=black]

    (5.5,0) node {$(G_2,A_1)$}
    (7,0) 
      circle [radius=.1] --node [midway,above] {$6$}
(8,0) 

(12,0) node {$\Rightarrow$}

(14,0) circle [radius=.1] --
(15,0)
      ;

\draw

(8,0) 
      circle [radius=.1]
      (15,0)
       circle [radius=.1] --
       (16,0)
       circle [radius=.1] --
       (17,0)
       circle [radius=.1] --
       (18,0)
       circle [radius=.1]
       
       ;
\end{tikzpicture}
\end{center}
The resulting graph for $(F_4,B_2)$ lacks the symmetry to come from a label free case via procedure~\ref{BU}.  It will follow from the case-by-case study for type $B$ below that $\cG(F_4/B_2,\le_B)$ is also different from those cases. 
The other right-hand sides are the (original) bw-Coxeter graphs, invariant under procedure~\ref{BU}, of respectively $(E_6,D_5)$, $(D_5,A_3)$ and $(A_5,A_4)$. 

The remaining Coxeter pairs with a labelled edge linking to the parabolic subgroup are

(a) $(B_n,A_{n-1})$ for $n\ge2$,  

(b) $(B_{m+n},P\times  A_{n-1})$ for some parabolic subgroup $P<A_{m-1}$ for the canonical inclusion $A_{m-1}\times A_{n-1}<B_{m+n}$ for $n\ge 2$ and $m\ge 1$, 

(c) $(B_n,B_{n-1})$ for $n\ge 3$,

(d) $(B_{m+n}, Q\times B_{n-1})$ for some parabolic subgroup $Q<A_{m}$ for the canonical inclusion $A_{m}\times B_{n-1}<B_{m+n}$ for $n\ge 2$ and $m\ge 1$.

It is again straightforward to calculate the corresponding graphs $\cG(W^J,\le_B)$, for instance using Lemma~\ref{unique}. For (a) the graph $\cG(W^J,\le_B)$ is the (original) bw-Coxeter graph of $(D_{n+1},A_{n})$ if $n\ge 3$, and of $(A_{3},A_{2})$ if $n=2$. For (c) it is the (original) bw-Coxeter graph of $(A_{2n-1},A_{2n-2})$. As an immediate extension of that, the cases (b) and (d) yield isomorphisms of bw-Coxeter graphs as displayed in the lemma. 
\end{proof}

\begin{proof}[Proof of Theorem~\ref{Thm3}]
That the listed pairs lead to isomorphic posets is well-known and easily checked. It remains to show that the cases in Lemma~\ref{LemNew} do not lead to isomorphic posets. Using the values in appendix~\ref{App}, we find
$$\Le(F_4/A_2)-\Le(D_5/A_3)=7,\quad\Le(F_4/B_3)-\Le(E_6/D_5)=-1,$$
$$\Le(B_{m+n}/(P\times  A_{n-1}))-\Le(D_{m+n+1}/(P\times A_n))=-m,\quad\mbox{and}$$
$$\Le(B_{m+n}/(Q\times B_{n-1}))-\Le(A_{m+2n-1}/(Q\times A_{2n-2}))=\frac{1}{2}m(m+1).$$
Consequently, the lengths of the posets never agree,
which concludes the proof. \end{proof}

\section{Category $\cO$}\label{SecO}
In this section we work over the field $\mC$ of complex numbers.
\subsection{Reduction to simple Lie algebras}

\subsubsection{}\label{SetUpDelO}Fix a reductive Lie algebra $\fg$ with Cartan and Borel subalgebras $\fh\subset\fb$. Assume furthermore that we have a non-trivial Lie algebra decomposition $\fg=\fg_1\oplus \fg_2$. This defines Borel and Cartan subalgebras $\fh_i\subset\fb_i\subset\fg_i$, for $1\le i\le 2$ and a factorisation of Coxeter groups $W=W_1\times W_2$. Fix also an antidominant weight $\lambda\in\fh^\ast$, which restricts to weights $\lambda^i=\lambda|_{\fh_i}$. For $1\le i\le 2$, we abbreviate $\cO_{\lambda_i}(\fg_i,\fb_i,\fh_i)$ to $\cO^i$ and $\cO_{\lambda}(\fg,\fb,\fh)$ to $\cO$. 

\subsubsection{} Following \cite{BGG}, we can define a projective object $P^i_\mu$ in $\cO^i$ for each $\mu$ in $W_i\cdot\lambda_i$, such that we have a canonical isomorphism
\begin{equation}\label{caniso}\Hom_{\fg_i}(P^i_{\mu},M)\;\stackrel{\sim}{\to}\; \Hom_{\fh_i}(\mC_\mu,M),\qquad\mbox{for every $M\in\cO^i$}.\end{equation}
Note that the \eqref{caniso} remains valid for $\fg$-modules $M$ in the ind-completion $\Ind \cO^i$, which is defined as in \ref{DefO} but without the condition of finite generation.
Isomorphism \ref{caniso} leads to $\mC$-linear equivalences
$$\Hom_{\fg_i}(\oplus_\mu P^i_\mu,-):\,\cO^i\,\stackrel{\sim}{\to}\,A_i\mbox{-mod},\qquad\mbox{with}\quad A_i:=\End_{\fg_i}(\oplus_\mu P^i_\mu)^{\op}.$$

Each $\kappa\in W\cdot\lambda$ can be written uniquely a sum $\kappa^1+\kappa^2$ with $\kappa^i\in W_i\cdot\lambda^i$.
It follows from a direct calculation, using the extension of \eqref{caniso} to $\Ind\cO^i$, that the tensor product $P_\kappa:=P^1_{\kappa^1}\otimes_{\mC}P^2_{\kappa^2}$ is a $\fg$-module which induces an isomorphism
$$\Hom_{\fg}(P_{\kappa},N)\;\stackrel{\sim}{\to}\; \Hom_{\fh}(\mC_\kappa,N),\qquad\mbox{for every $N\in\cO$}.$$ 
In conclusion, we find a $\mC$-linear equivalence
$$\cO\,\simeq\,(A_1\otimes A_2)\mbox{-mod}.$$
In order to formulate this result without mentioning of the auxiliary algebras $A_i$, we can employ Deligne's tensor product of abelian $\mC$-linear categories $-\boxtimes-$, see \cite[Section~5]{De}.
\begin{prop}
For a reductive Lie algebra $\fg=\fg_1\oplus\fg_2$ as in \ref{SetUpDelO}, we have a $\mC$-linear equivalence
$$\cO_\lambda(\fg,\fb,\fh)\;\simeq\; \cO_{\lambda_1}(\fg_1,\fb_1,\fh_1)\boxtimes \cO_{\lambda_2}(\fg_2,\fb_2,\fh_2).$$
\end{prop}

We can reformulate this result in terms of Weyl groups, based on the discussion in \ref{ResSoergel}.

\begin{corollary}\label{CorOSimp}
Let $W$ be a finite Weyl group, with factorisation $W=\prod_{i=1}^dW_i$ into irreducible Coxeter groups and generating set $S=\sqcup_{i=1}^d S_i$. For any $J\subset S$ with $J_i=J\cap S_i$, we have a $\mC$-linear equivalence
$$\cO(W,W_J)\;\simeq\;\cO(W_1,{W_1}_{J_1})\boxtimes \cO(W_2,{W_2}_{J_2})\boxtimes\,\cdots\,\boxtimes \cO(W_d,{W_d}_{J_d}).$$
\end{corollary}

\subsection{Main result}

\begin{theorem}\label{Thm4}
Let $(W,S)$ and $(U,T)$ be finite Weyl groups and take $J\subset S$ and $K\subset T$. Consider the factorisations into irreducible Weyl groups
\begin{equation}W\;\stackrel{J}{=}\;\prod_{i=1}^{n_1}W_i\quad\mbox{and}\quad U\;\stackrel{K}{=}\;\prod_{i=1}^{n_2}U_i.\label{Fact2Weyl}\end{equation} The following properties are equivalent
\begin{enumerate}
\item There exists a $\mC$-linear equivalence $\cO(W,W_J)\simeq \cO(U,U_K)$.
\item There exists an order isomorphism $(W^J,\le_B)\simeq (U^K,\le_B)$.
\item We have $n_1=n_2$ and can (re)order the factors in \eqref{Fact2Weyl}
such that we have order isomorphisms $(W_i^{J_i},\le_B)\simeq (U_i^{K_i},\le_B)$ for all $i$.
\item We have $n_1=n_2$ and can (re)order the factors in \eqref{Fact2Weyl}
such that we have $\mC$-linear equivalences $\cO(W_i,{W_i}_{J_i})\simeq \cO(U_i,{U_i}_{K_i})$ for all $i$.
\end{enumerate}
\end{theorem}
\begin{proof}
That (1) implies (2) is observed in \ref{CorO}. That (2) implies (3) is a special case of Theorem~\ref{Thm2}. That (4) implies (1) follows from Corollary~\ref{CorOSimp}.

In order to show that (3) implies (4) it suffices to prove that for two irreducible finite Weyl groups $W$ and $U$ an isomorphism $(W^J,\le_B)\simeq (U^K,\le_B)$ implies an equivalence $\cO(W,W_J)\simeq \cO(U,U_K)$. In other words, we need to prove an equivalence for each of the order isomorphisms (1)-(4) listed in Theorem~\ref{Thm3}. The equivalences for \ref{Thm3}(3) and (4) follow from the explicit description in \cite[Section~5]{St}.
The equivalences for \ref{Thm3}(1) and (2) are well known. For instance, the Koszul duals (see \cite{BGS}) of the two algebras describing the respecting blocks are observed to be Morita equivalent in \cite[\S 1.7]{ES}.
\end{proof}

\subsection{Outlook: Generalisations of category $\cO$}

\subsubsection{Kac-Moody algebras }\label{KM} Theorem~\ref{ThmQH} extends immediately to the setting of {\em upper finite highest weight categories} of \cite[Section~3.3]{BS}. Concretely, if we have a simple preserving duality on such a category, then \eqref{BGG} is still valid and we have maximal elements in the poset. This is all that is required for the proof. This means that our results could be relevant also in the study of category $\cO$ for Kac-Moody algebras, see \cite[Section~5.2]{BS}.

\subsubsection{Lie superalgebras} Theorem~\ref{ThmQH} does {\bf not} extend to the setting of {\em essentially finite highest weight categories} of \cite[Section~3.2]{BS}. The proof does not extend, due to lack of maximal elements in the poset. A concrete counter example to Theorem~\ref{ThmQH} in that setting is given by category $\cO$ for a basic classical Lie superalgebra, which is an essentially finite highest weight category with a simple preserving duality, see~\cite[Section~5.5]{BS}. Different (non-conjugate) Borel subalgebras with same underlying even Borel subalgebra can lead to non-equivalent highest weight structures.

Some results on the classification of the blocks in category $\cO$ for $\mathfrak{gl}(m|n)$ are obtained in \cite{BG, CS}. We hope that the full classification for Lie algebras can help in the classification problem for $\mathfrak{gl}(m|n)$.

\subsubsection{Root reductive Lie algebras} 
Similarly to \ref{KM}, category $\cO$ for $\mathfrak{gl}(\infty)$ has blocks which are upper finite highest weight categories in the sense of \cite[Section~3.3]{BS}. Our methods hence apply to that case. As a concrete example, the methods in Section~\ref{SecUnique} show that the blocks compared in \cite[Lemma~4.7.2]{CP} are not equivalent.
\appendix

\section{The finite Weyl groups}\label{App}

We list the Coxeter graphs for all finite (irreducible) Weyl groups. We also include the length $\ell(w_0)$ of the longest element.

\vspace{3mm}

\begin{center}

\begin{tikzpicture}
\draw[fill=black] 
(-1,0) node {$A_n$}

(0,0)                         
      circle [radius=.1] node [above] {} --
(1,0) 
      circle [radius=.1] node [above] {} --
(2,0) 
      circle [radius=.1] node [above] {}      
(3.5,0)
    node {$\cdots$}
    (5,0)      
    circle [radius=.1] node [above] {} --
      (6,0)      
    circle [radius=.1] node [above] {}            
    (8,0)
    node {$n\ge 1$}
    
    (12,0) node {$\ell(w_0)=\frac{n(n+1)}{2}$}
;
\end{tikzpicture}

\vspace{2mm}

\begin{tikzpicture}
\draw[fill=black] 
(-1,0) node {$B_n$}

(0,0)                         
      circle [radius=.1]  --
(1,0) 
      circle [radius=.1] --
(2,0) 
      circle [radius=.1]     
(3,0)
    node {$\cdots$}
    (4,0)      
    circle [radius=.1] --
    (5,0)      
    circle [radius=.1] --node [midway,above] {$4$}
      (6,0)      
    circle [radius=.1]       
    (8,0)
    node {$n\ge 2$}     
    
     (12,0) node {$\ell(w_0)=n^2$}
;
\end{tikzpicture}

\vspace{2mm}

\begin{tikzpicture}
\draw[fill=black] 
(-1,0) node {$D_n$}

(0,0)                         
      circle [radius=.1] node [above] {} --
(1,0) 
      circle [radius=.1] node [above] {}
(2,0) 
    node {$\cdots$}
(3,0)      
    circle [radius=.1] node [above] {} --
      (4,0)      
    circle [radius=.1] node [above] {}     
    
(4,0) 
      circle [radius=.1] node [above] {}      
            
(4,0) --++ (30:1.15)
      circle [radius=.1] 
      
(4,0) --++ (-30:1.15)
      circle [radius=.1]      
      (8,0)
    node {$n\ge 4$} 
      (12,0) node {$\ell(w_0)=n^2-n$}

; 

\end{tikzpicture}

\vspace{2mm}

\begin{tikzpicture}
\draw[fill=black] 
(-1,0) node {$E_6$}

(0,0)                         
      circle [radius=.1]  --
(1,0) 
      circle [radius=.1] node [above] {} --
(2,0) 
      circle [radius=.1] node [above] {}    --  
    (3,0)      
    circle [radius=.1] node [above] {} --
      (4,0)      
    circle [radius=.1] 
    (2,0) --++(90:1)         circle [radius=.1]    
    
    (12,0) node {$\ell(w_0)=36$}
;
\end{tikzpicture}

\vspace{2mm}

\begin{tikzpicture}
\draw[fill=black] 
(-1,0) node {$E_7$}

(0,0)                         
      circle [radius=.1]  --
(1,0) 
      circle [radius=.1] node [above] {} --
(2,0) 
      circle [radius=.1] node [above] {}    --  
    (3,0)      
    circle [radius=.1] node [above] {} --
      (4,0)      
    circle [radius=.1] --
    (5,0)      
    circle [radius=.1]
    (2,0) --++(90:1)         circle [radius=.1]    
        (12,0) node {$\ell(w_0)=63$}

;
\end{tikzpicture}

\vspace{2mm}

\begin{tikzpicture}
\draw[fill=black] 
(-1,0) node {$E_8$}

(0,0)                         
      circle [radius=.1]  --
(1,0) 
      circle [radius=.1] node [above] {} --
(2,0) 
      circle [radius=.1] node [above] {}    --  
    (3,0)      
    circle [radius=.1] node [above] {} --
      (4,0)      
    circle [radius=.1] --
      (5,0)      
    circle [radius=.1] --
      (6,0)      
    circle [radius=.1] 

    (2,0) --++(90:1)         circle [radius=.1]    
        (12,0) node {$\ell(w_0)=120$}

;
\end{tikzpicture}

\vspace{2mm}

\begin{tikzpicture}
\draw[fill=black] 
(-1,0) node {$F_4$}

(0,0)                         
      circle [radius=.1]  --
(1,0) 
      circle [radius=.1] --node [midway,above] {$4$}
(2,0) 
      circle [radius=.1] --    
(3,0)
    circle [radius=.1] 
       (12,0) node {$\ell(w_0)=24$}
    ;\end{tikzpicture}
    
    \begin{tikzpicture}     
    \draw[fill=black] 
    (6,0) node {$G_2$}
    (7,0) 
      circle [radius=.1] --node [midway,above] {$6$}
(8,0) 
      circle [radius=.1]
      
         (19,0) node {$\ell(w_0)=6$}
;
\end{tikzpicture}

\end{center}

\subsection*{Acknowledgement}
The research was supported by the ARC grant DE170100623. The author thanks Chih-Whi Chen and Geordie Williamson for interesting discussions.

\end{document}